\documentclass[a4paper,12pt]{amsart}
\usepackage[utf8]{inputenc}
\usepackage[T1]{fontenc}
\usepackage[english]{babel}
\usepackage{amsmath}
\usepackage{amsthm}
\usepackage{amssymb}
\usepackage{mathtools}
\usepackage{nicefrac}
\usepackage{hyperref}
\usepackage{adjustbox}
\usepackage{graphicx}
\usepackage{braket}
\usepackage{subfig}
\usepackage{csquotes}
\usepackage[nobysame]{amsrefs}

\newcommand{\R}{{\mathbb R}}
\newcommand{\N}{{\mathbb N}}
\newtheorem{teo}{Theorem}[section]
\newtheorem{lemma}{Lemma}[section]
\newtheorem{cor}{Corollary}[section]
\newtheorem{oss}{Remark}[section]

\title[Maximum principle for higher order operators]{Maximum principle for higher order operators in general domains}
\dedicatory{In loving memory of Louis Nirenberg}
\author[D.~Cassani]{Daniele Cassani$^\text{1}$}
\author[A.~Tarsia]{Antonio Tarsia}

\address[D.~Cassani]{\newline\indent Dip. di Scienza e Alta Tecnologia
\newline\indent
Universit\`{a} degli Studi dell'Insubria
\newline\indent and
\newline\indent RISM--Riemann International School of Mathematics
\newline\indent Villa Toeplitz, Via G.B. Vico, 46 -- 21100 Varese}
\email{\href{mailto:Daniele.Cassani@uninsubria.it}{Daniele.Cassani@uninsubria.it}}

\address[A.~Tarsia]{\newline\indent Dip. di Matematica
\newline\indent
Universit\`{a} di Pisa
\newline\indent Largo B. Pontecorvo, 5  -- 56127 Pisa}
\email{\href{mailto:antonio.tarsia@unipi.it}{antonio.tarsia@unipi.it}}

\thanks{(1) Corresponding author: \texttt{daniele.cassani@uninsubria.it}}

\date{\today}

\begin{document}

\begin{abstract} 
We first prove De Giorgi type level estimates for functions in $W^{1,t}(\Omega)$, $\Omega\subset\R^N$, with $t>N\geq 2$. This augmented integrability enables us  to establish a new Harnack type inequality for functions which do not necessarily belong to De Giorgi's classes as obtained in Di Benedetto--Trudinger \cite{DiBenedettoTrudinger} for functions in $W^{1,2}$. As a consequence, we prove the validity of the strong maximum principle for uniformly elliptic operators of any even order, in fairly general domains in dimension two and three, provided second order derivatives are taken into account.
\end{abstract}

\maketitle

\section{Introduction}

\noindent One of the most powerful tools in the study of partial differential equations and nonlinear analysis is without any doubts the Maximum Principle (MP in the sequel). It turns out to be fundamental in obtaining existence, uniqueness and regularity results in the theory of linear elliptic equations, as well as to establish qualitative properties of solutions to nonlinear equations. We mainly refer to \cite{Protter_Weinberger} for classical results and historical development, where suitable applications also to the parabolic and hyperbolic cases are discussed. Let us merely mention that the roots of MP date back two centuries in the work of Gauss on harmonic functions, up to the ultimate version of Hopf \cite{Hopf}, and then further extended in the seminal work of Nirenberg \cite{Nirenberg}, Alexandrov \cite{Alexandrov} and Serrin \cite{Serrin}, within the foundations of modern theory of PDEs. 

\noindent The underlying idea is simple: positivity of a suitable set of derivatives of a function induces positivity of the function itself. This is elementary true for real functions of one variable which vanish at the endpoints of an interval where $-u''(x)\geq 0$ and the validity can be extended to second order uniformly elliptic operators for which a prototype is the Laplace operator:
\begin{equation}\label{laplace}
\begin{cases}
-\Delta u=f,\quad &\text{ in } \Omega\subset \mathbb{R}^N, \quad N\geq 2\\
u=0, &\text{ on } \partial\Omega 
\end{cases}
\end{equation}  
for which we have 
$$f\geq 0\Longrightarrow  u\geq 0\text{ in } \Omega\ .$$
\noindent Surprisingly, this is no longer true when considering higher order elliptic operators such as the biharmonic operator $\Delta^2$:
\begin{equation}\label{biharmonic}
\begin{cases}
\Delta^2 u=f,\quad &\text{ in } \Omega\\
u=\frac{\partial u}{\partial \nu}=0, &\text{ on } \partial\Omega\ . 
\end{cases}
\end{equation}  
Indeed, in this case in general one has 
$$
f\geq 0\not\Longrightarrow u\geq 0\text{ in } \Omega\ .
$$ 
This is a well known fact as long as the domain $\Omega$ is not a ball, for which the positive Green function was computed by Boggio \cite{boggio} and which keeps on being positive for slight deformations of the ball \cite{Sassone}. As deeply investigated in \cite{GazzolaGrunauSweers10} and references therein, the lack of the positivity preserving property is due to the appearance of sets carrying small Hausdorff measure (see \cite{Grunau_Robert}) where $u<0$ and apparently without robust physical motivations. Recently the loss of the MP has been established in \cite{AJS} also in the case of higher order fractional Laplacians. This paper is a step forward a better understanding of this phenomena and at the same time gives some general principle in order to recover the validity of the MP in the higher order setting. 

\noindent Let us briefly recall some physical interpretation of  \eqref{laplace}--\eqref{biharmonic}. Indeed, \eqref{laplace} is modeling, among many other things, a membrane whose profile is $u$ which deflects under the charge load $f$ and clamped along the boundary $\partial\Omega$. This is the case in which tension forces prevale on bending forces which can be neglected because of the ``thin'' membrane. However, the model does not suite the case of a ``thick'' plate in which bending forces have to be taken into account. Here higher order derivatives come into play which yield \eqref{biharmonic}. As one expects for \eqref{laplace}, and there this is true by the MP, upwards pushing of a plate, clamped along the boundary, should yield upwards bending: this is false for \eqref{biharmonic} in contrast to some heuristic evidences in applications (see e.g.~\cite{Kempe} and references therein). 

\noindent Our point of view here, roughly speaking, is that approaching the boundary, where the bending energy carries some minor effect because of the clamping condition, tensional forces can not be neglected for which the contribute of lower order derivatives may restore the validity of the MP. As a reference example, consider the following simple model: 
\begin{equation}\label{cassani-tarsia}
\begin{cases}
\Delta^2 u-\gamma \Delta u=f,\quad &\text{ in } \Omega\subset \mathbb{R}^N, \, \gamma\geq 0\\
u=\frac{\partial u}{\partial \nu}=0, &\text{ on } \partial\Omega \ .
\end{cases}
\end{equation}
Clearly for $\gamma=0$ one has \eqref{biharmonic} whence formally as $\gamma\to\infty$, in a sense one may expect that \eqref{cassani-tarsia} inherits some properties of \eqref{laplace}. 

\noindent As we are going to see, this is the case and for the more accurate model \eqref{cassani-tarsia} surprisingly the MP holds true, for fairly general domains, provided $\gamma\geq \gamma_0>0$, which is essentially given in terms of Sobolev and Poincar\'e best constants. Let us state our main result in the case of \eqref{cassani-tarsia} though it extends to cover the general case of uniformly elliptic operators of any even order, see Corollary \ref{higher_order_cor}.

\begin{teo}\label{main}
Let $\Omega \subset \R^N, N=2,3$, be an open and bounded set, with sufficiently smooth boundary and which satisfies the interior sphere condition. Let $u\in H^2_0(\Omega)$ be a weak solution to \eqref{cassani-tarsia}, where $f\in L^2(\Omega)$, $f\geq 0$ in $\Omega$ and $|\{x: f(x)>0 \}|>0$. Then, there exists $\gamma_0>0$ (which depends on $f$, Sobolev and Poincar\'e best constants), such that for $\gamma >\gamma_0$ one has $u>0$ in $\Omega$.
\end{teo}

\noindent As a consequence of Theorem \ref{main}, the operator $\Delta^2-\gamma\Delta$, which in addition to \eqref{biharmonic} contains the contribute of lower order derivatives, turns out to be a more natural extension of \eqref{laplace} to the higher order setting.

\medskip 

\noindent \textbf{Overview}. In Section \ref{preliminaries} we prove some preliminary estimates which will be the key ingredient to prove in Section \ref{harnack} a new Harnack type inequality. Indeed, in the higher order case, it is well known how truncation methods fail \cite{GazzolaGrunauSweers10}. Our approach here is to demand some extra integrability on the function entering the Harnack inequality in place of being solution to a PDE, which usually yields Caccioppoli's inequality and the solution belongs to the corresponding De Giorgi class. In \cite{DiBenedettoTrudinger} the authors prove a Harnack type inequality just for functions with membership in some De Giorgi classes. Here we drop this assumption though we assume more regularity in terms of integrability which however enables us to prove De Giorgi type pointwise level estimates. In Section \ref{higher_order} we apply the results obtained to prove the strong maximum for polyharmonic operators of any order, which contain lower order derivatives, in sufficiently smooth bounded domains which enjoy the interior sphere condition. This is done by a limiting procedure starting from compactly supported functions and then extending the results and estimates to the solutions of higher order PDEs subject to Dirichlet boundary conditions. Those boundary conditions are in a sense the natural ones as the higher order operator in this case does not decouple into powers of a second order operator. In one hand the result we obtain is a first step towards the investigation of qualitative properties of higher order nonlinear PDEs, such as uniqueness, optimal regularity, symmetries and concentration phenomena \cite{berchio, CassaniSchiera19, GidasNiNirenberg79, Martinazzi, MM, PV, Pucci_Serrin, Sweers}. On the other hand, we are confident the tools introduced here may reveal useful also in different contexts, such as parabolic problems, in the study of the sign of solutions to quasilinear equations and in the higher order fractional Laplacian setting. 

\noindent This research started in 2010 when Theorem \ref{main} was settled by the first named author in the form of conjecture in a conference in Pisa. New advances towards the results in this paper have been made in 2014 during the first visit of Louis Nirenberg in Varese, then in New York 2015, Pisa 2016 and Varese again in 2017 (his last trip), occasions in which Louis has further stimulated this research during long discussions of which we keep nostalgic memories. Goodbye Louis!

\medskip

\noindent \textbf{Notation}. In the sequel we will use the following basic definitions:
\begin{itemize}
\item $B(x_0,r)$ denotes the ball in $\R^N$ of center $x_0$ and radius $r$;
\item $\omega_N$ is the volume of the unit ball in $\mathbb{R}^N$;
\item $d_{\Omega}$ denotes the diameter of the bounded set $\Omega$ in $\R^N$;
\item $|\,\cdot\,|$ applied to sets denotes the Lebesgue measure in $\R^N$ otherwise it is the Euclidean norm in $\R^N$ with scalar product $(\cdot \, , \, \cdot)$;
\item $A^+(x_0,k,r):=\{ x:\,x\in B(x_0,r),\,\,u(x)>k\}$;
\item  $(u-k)^+ := \max \{u-k,\,0\}$;
\item $\Omega$ satisfies the \textit{interior sphere condition} if for all $x \in \partial \Omega$ there exists $y \in \Omega $ and $r_0>0$ such that $B(y,r_0)\subset\Omega$ and $x \in \partial B(y,r_0)$;
\item $c$ and $C$ denote positive constants which may change from line to line and which do not depend on the other quantities involved unless explicitly emphasized;
\item $W^{m,p}(\Omega)$ is the standard Sobolev space endowed with the norm $\|\cdot\|_{m,p}^p=\sum_{0\leq |\alpha|\leq m}\|D^{\alpha}u\|_p^p$;
\item $W^{m,p}_0(\Omega)$ is the completion of smooth compactly supported functions with respect to the norm $\|\cdot\|_{m,p}$;
\item the critical Sobolev exponent $p^*:=\frac{Np}{N-mp}$, $1<p<N/m$.
\end{itemize}
\section{Preliminaries}\label{preliminaries}
\noindent Let $\Omega\subset\R^N$, $N\geq 2$, be an open bounded set with sufficiently smooth boundary. The following holds true 
\begin{lemma}\label{lemma0}
Let $u \in W^{1,t}(\Omega)$, $t>N$ and $1<s<N$. Then there exists $c(s,t)>0$ such that for all $k \in \R$, $x_0\in \Omega$ and $\rho \in (0,r)$ where $0<r \,<\,dist(x_0,\partial \Omega)$ the following holds
\begin{multline}\label{0.1}
\int_{A^+(x_0,k,\rho)}  (u-k)^{s^*}\,\,\,dx\leq \frac{c(s,t,N)}{(r \,-\,\rho)^{s^*}}\,\,|A^+(x_0,k,r)|^{(1-\frac{s}{t}) \frac{s^*}{s}}\\
\cdot \left[ \int_{A^+(x_0,k,r)} (u-k)^t\,\,dx\,+
\,r^t\,\int_{A^+(x_0,k,r)} |\nabla u|^t\,\,dx\right]^{\frac{s^*}{t}}\ .
\end{multline}

\end{lemma}

\begin{proof} Consider a standard cut-off function $\Theta \in C^{\infty}_0(\R^N)$ given by 
\begin{eqnarray}
\Theta(x)=\left\{
\begin{array}{l}
1, \,\,\,x \in B(x_0,\rho)\\
\\
0,\,\,\,x \notin B(x_0,r),
\end{array}
\right.
\end{eqnarray}
such that $0\leq \Theta (x) \leq 1$ and $\displaystyle  |\nabla \Theta (x) | \leq \frac{c}{r-\rho}$.

\noindent As $W^{1,t}(\Omega)\hookrightarrow W^{1,s}(\Omega)$, one has from Sobolev's emedding and H\"older's inequality
\begin{eqnarray*}
&&\int_{A^+(x_0,k,\rho)}  (u-k)^{s^*}\,\,\,dx\,\leq \displaystyle \int_{A^+(x_0,k,r)}  |(u-k)\Theta(x)|^{s^*}\,\,\,dx\,\\
&&\leq \,c(s)\,\left[\int_{A^+(x_0,k,r)} |\nabla [(u-k)\,\,\Theta]|^s\,\,dx\right]^{\frac{s^*}{s}}\\
&&\leq \,c(s)\,|A^+(x_0,k,r)|^{[1-\frac{s}{t}]\frac{s^*}{s}} \,\,\left[\int_{A^+(x_0,k,r)} |\nabla [(u-k)\,\,\Theta]|^t\,\,dx\right]^{\frac{s^*}{t}} \\ 
&&\leq \,c(s,t)\,|A^+(x_0,k,r)|^{[1-\frac{s}{t}]\frac{s^*}{s}} \\
&&\quad \cdot\left[ \displaystyle \frac{c}{(r-\rho)^t}\,\int_{A^+(x_0,k,r)}  (u-k)^t\,\,\,dx\,+\,\int_{A^+(x_0,k,r)} |\nabla u |^t\,\,\,dx\right]^{\frac{s^*}{t}}
\end{eqnarray*}
\end{proof}
\begin{oss}
The condition $0<r<dist(x_0, \partial \Omega)$, namely that $x_0$ lies in the interior of $\Omega$, is crucial to extend to the whole $\R^N$ 
the function $(u-k)\Theta$. Therefore when $x_0$ approaches $\partial \Omega$, necessarily $r=r(x_0)$ tends to zero.
\end{oss}
\begin{lemma}\label{lemma1}
Let $u \in W^{1,t}(\Omega)$, $t>N\geq 2$ and $1<s<N$. Let $l$, $k\in \R$ such that $l>k$, $x_0\in \Omega$ and $r<dist(x_0,\partial\Omega)$. Then for all $\rho \in(0,r)$ one has 
\begin{eqnarray}\label{1.1}
&&\int_{A^+(x_0,l,\rho)}  (u-l)^2\,dx \leq \frac{c(t)\,\,|A^+(x_0,k,r)|^{\beta}}{(r-\rho)^{\frac{2(p-1)}{p}}}  \left[
\int_{A^+(x_0,k,r)}  (u-k)^{2}\,dx\,\right]^{\frac{1}{p}} 
\\
\nonumber &&\quad  \cdot \,\left[  \int_{A^+(x_0,k,r)} (u-k)^t\,\,dx\, +\,r^t\,\int_{A^+(x_0,k,r)} |\nabla u|^t\,\,dx \right]^{\frac{2(p-1)}{pt}}\ ,
\end{eqnarray}
where $\displaystyle \beta = 1- \frac{2}{q}+\left( 1-\frac{s}{t}\right) \frac{s^*}{s} \,\frac{2p-q}{pq}$,  
$s=\displaystyle\frac{2qN(p-1)}{N(2p-q)+2q(p-1)}$ and $2<q<2p$, $p>1$.
\end{lemma}

\begin{proof}

Let $x_0\in \Omega$, $r<dist(x_0,\partial\Omega)$
and for simplicity let us write $A^+(k,r)$ in place of $A^+(x_0,k,r)$.

\noindent For $l,k \in \R$  and $\rho \in (0,r)$, since $A^+(l,\rho)\subset A^+(k,\rho)$ we have 
 \begin{eqnarray}
\int_{A^+(l,\rho)}  (u-l)^2\,\,\,dx \leq \int_{A^+(k,\rho)}  (u-k)^2\,\,\,dx. 
\label{1.2}
\end{eqnarray}
Let $q>2$ for which one has  
\begin{eqnarray}
\int_{A^+(k,\rho)}  (u-k)^2\,\,\,dx\,\leq \,|A^+(k,\rho)|^{1-\frac{2}{q}}\left[  \int_{A^+(k,\rho)}  (u-k)^q\,\,\,dx \right]^{\frac{2}{q}} \ .
\label{1.3}
\end{eqnarray}
Let now $p>1$ and $2<q<2p$ and estimate by H\"older's inequality
\begin{multline}\label{1.4}
\int_{A^+(k,\rho)}  (u-k)^q\,\,\,dx \,\,\ \\
\leq \,\, \left[  \int_{A^+(k,\rho)}  (u-k)^2\,\,\,dx \right]^{\frac{1}{p}} 
 \left[  \int_{A^+(k,\rho)}  (u-k)^{\frac{2q(p-1)}{(2p-q)}}\,\,\,dx \right]^{\frac{2p-q}{2p}}\\
\\
\leq \frac{c(p,q,t)}{(r-\rho)^{\frac{q(p-1)}{p}}}\,\,|A^+(k,r)|^{(1-\frac{s}{t})\frac{q(p-1)}{sp}} \left[  \int_{A^+(k,\rho)}  (u-k)^2\,\,\,dx \right]^{\frac{1}{p}} \,\,
 \\
\quad  \cdot
\left[  \int_{A^+(k,r)} (u-k)^t\,\,dx\, +\,r^t\,\int_{A^+(k,r)} |\nabla u|^t\,\,dx \right]^{\frac{2(p-1)}{tp}} \ ,
\end{multline}
where in the last inequality we have used Lemma \ref{lemma0} with $s^*=\frac{2q(p-1)}{2p-q}$. Combine (\ref{1.3}) and (\ref{1.4}) to get 
\begin{multline}\label{1.5}
\int_{A^+(k,\rho)}  (u-k)^2\,\,\,dx \,\,\ \\
\leq\,\,\frac{c(p,q,t,N)}{(r-\rho)^{2\frac{p-1}{p}}}\,\,|A^+(k,r)|^{(1-\frac{2}{q})\frac{q}{2p}+(1-\frac{s}{t})\frac{2(p-1)}{sp}} \left[  \int_{A^+(k,\rho)}  (u-k)^2\,\,\,dx \right]^{\frac{1}{p}} \,\,\\
 \quad \cdot
\left[  \int_{A^+(k,r)} (u-k)^t\,\,dx\, +\,r^t\,\int_{A^+(k,r)} |\nabla u|^t\,\,dx \right]^{\frac{2(p-1)}{tp}} \ .
\end{multline}
\end{proof}

\noindent In what follows we will use the following result from \cite{ACM} in order to prove a version of the well known Poincar\'e inequality. 
\begin{teo}[Theorem A. 28, p. 184 in \cite{ACM}]
Let $u \in W^{1,1}(B_r)$, auch that $u \geq 0$ and $|\{x:\,u(x)=0 \}|\geq \frac{|B_r|}{2}$. Then
\begin{eqnarray}
\left(\int_{B_r} u^{1^*}\,\,\,dx \right)^{\frac{1}{1^*}} \leq \,c\, \int_{B_r}|\nabla u|\,\,\,dx,
\label{MaggPoincare1}
\end{eqnarray}
where $c=c(N)$ depends only on the dimension  $N$.
\label{TeoPoinacare1}
\end{teo}

\begin{lemma}\label{lemmapoincare}
Let $u \in W^{1,p}(B_r)$, $p> 1$, be such that $|\{x:\,u(x)=0 \}|\geq \frac{|B_r|}{2}$. Then the following holds
\begin{eqnarray}
\left(\int_{B_r} |u|^p\,\,\,dx \right)^{\frac{1}{p}} \leq \,c\, \omega_N\,p\,\frac{N-1}{N}\,r \, \left(\int_{B_r}|\nabla u|^p\,\,\,dx \right)^{\frac{1}{p}},
\label{MaggPoincare2}
\end{eqnarray}
where $c=c(N)$ is the constant in (\ref{MaggPoincare1}).
\end{lemma} 
\begin{proof}
Apply Theorem  \ref{TeoPoinacare1} to the function $|u|^{p}$, to get 
\begin{multline*}
\int_{B_r} |u|^{p}\,\,\,dx  =\int_{B_r} \left(|u|^{p\frac{N-1}{N}}\right)^{\frac{N}{N-1}}\,\,\,dx\\
\leq \,c\, \left[\int_{B_r}|\nabla\left( |u|^{p\frac{N-1}{N}}\right)|\,\,\,dx \right]^{\frac{N}{N-1}}\\
  = c \left[\int_{B_r}p\frac{N-1}{N}\,|u|^{p\frac{N-1}{N}-1}  |\nabla u|\,dx \right]^{\frac{N}{N-1}}  \\
    =\,  c \,p^{\frac{N}{N-1}}\left(\frac{N-1}{N}\right)^{\frac{N}{N-1}} \left[\int_{B_r}|u|^{p\frac{N-1}{N}-1}  |\nabla u|\,dx \right]^{\frac{N}{N-1}}\\
  =  c \,p^{\frac{N}{N-1}}\left(\frac{N-1}{N}\right)^{\frac{N}{N-1}} \left[\int_{B_r}|u|^{p(\frac{N-1}{N}-\frac{1}{p})}  (|\nabla u|^p)^{\frac{1}{p}}\,\,1^{\frac{1}{N}}\,\,dx \right]^{\frac{N}{N-1}}\\
    \leq    
     c \,p^{\frac{N}{N-1}}\left(\frac{N-1}{N}\right)^{\frac{N}{N-1}} \left[\int_{B_r}|u|^p\,\,dx\right]^{(\frac{N-1}{N}-\frac{1}{p})\frac{N}{N-1}} \\
   \cdot \,  \left[\int_{B_r} |\nabla u|^p \,\,dx \right]^{\frac{1}{p}\frac{N}{N-1}}\,\,|B_r|^{\frac{1}{N-1}}\,.
\end{multline*}
Since 
\[
\left(\frac{N-1}{N}\,-\,\frac{1}{p} \right)\,\frac{N}{N-1} = 1 \,-\,\frac{1}{p} \,\frac{N}{N-1},
\]
we have 
\[
 \left[\int_{B_r}|u|^p\,\,dx\right]^{\frac{1}{p}\frac{N}{N-1}}\leq \,
     c \,p^{\frac{N}{N-1}}\left(\frac{N-1}{N}\right)^{\frac{N}{N-1}}\,\omega_N\,r^{\frac{N}{N-1}} \left[\int_{B_r} |\nabla u|^p \,\,dx \right]^{\frac{1}{p}\frac{N}{N-1}}
\]

\end{proof}

\section{A Harnack type inequality}\label{harnack}

\noindent Next we derive a De Giorgi type level estimate (see \cite{ACM,GiaquintaMartinazzi}) for functions $u \in W^{1,t}$, $t>N\geq 2$ which will be the key ingredient in establishing a new Harnack type inequality. Let us emphasize that in De Giorgi's theorem \cite{DeGiorgi}, level estimates hold for $u \in W^{1,2}$ which is a solution to a uniformly elliptic second order equation with bounded and measurable coefficients. As a consequence, Caccioppoli's inequality holds and $u \in W^{1,2}$ belongs to the corresponding so-called De Giorgi class. Later, Di Benedetto and Trudinger relaxed the framework and in \cite{DiBenedettoTrudinger} they merely assume $u \in W^{1,2}$ belonging to some De Giorgi class. Here, we further improve the setting, without requiring any of those previous assumptions, though demanding for some augmented integrability which turns out to be necessary, as it is well known, functions in $W^{1,N}(\Omega)$, $\Omega\subset\R^N$, may not be bounded. 

\begin{teo}\label{teorema1}
Let $u\in W^{1,t}(\Omega)$, $t>N\geq 2$, $\Omega \subset \R^N$ be open and bounded set with sufficiently smooth boundary $\partial \Omega$.
For all $k \in \R$, $y \in \Omega$, $r>0$ such that $r < dist(y,\partial \Omega)$ the following holds 
\begin{equation}\label{2.1}
\displaystyle \sup_{B(y,\frac{r}{2})}u\,\leq \, k\,+\,d,
\end{equation}
where 
\begin{multline*}
d  = \displaystyle   \frac{c}{r^{\frac{\xi(p-1)}{\eta p}}}\left( \int_{A^+(k,r)} |u(x)-k|^t\,\,\,dx\,
 +\,r^t\,\int_{A^+(k,r)} |\nabla u(x)|^t\,\,\,dx \right)^{\frac{\xi(p-1)}{tp \eta}} \\
    \cdot \displaystyle
    \left(\int_{A^+(k,r)} |u(x)-k|^2\,\,\,dx\right)^{\frac{\xi(\theta-1)}{2\eta}}|A^+(k,r)|^{\frac{\theta -1}{2}}
\end{multline*}
and where $c=c(t,\xi,\eta,\mu,p)$ is a positive constant.
\end{teo}

\begin{proof}
Let $y \in \Omega$ and let us write for simplicity $A^+(k,\rho)$ in place of $A^+(y,k,\rho)$. Moreover, let us set: 
\begin{eqnarray*}
I(l,\rho) & = & \displaystyle \int_{A^+(l,\rho)} |u-l|^2\,\,\,dx\ ,  \\
    &   &  \\
    M(r,k,t,p)& = & \displaystyle    c(t) \left( \int_{A^+(k,r)} |u-k|^t\,\,\,dx\,+\,r^t\,\int_{A^+(k,r)} |\nabla u|^t\,\,\,dx \right)^{\frac{2(p-1)}{pt}}.
\end{eqnarray*}
For all $l,k\in \R$, such that $l>k$ and for all $\rho \in (0,r)$, one has 
\begin{eqnarray}
|A^+(l,\rho)|\leq \frac{1}{(l-k)^2}\,I(k,\rho)
\label{2.3}
\end{eqnarray} 
and clearly $|A^+(l,\rho)|\leq |A^+(k,\rho)|$, for  $l>k$.

\noindent Set
\begin{eqnarray}
\Phi(l,\rho)\,=\,I(l,\rho)^{\xi}\,\,|A^+(l,\rho)|^{\eta},
\label{2.5}
\end{eqnarray}
then from \eqref{2.3} and \eqref{1.1} we have 
\begin{eqnarray}
\Phi(l,\rho)\,\leq\, \frac{1}{(r-\rho)^{2\xi\frac{p-1}{p}}\,\,(l-k)^{2\eta}}\,\Phi(k,r)^{\theta} \,M(r,k,t,\rho)^{\xi},
\label{2.6}
\end{eqnarray}
where $\eta,\xi,\theta>0$ satisfy the following algebraic equations
\begin{eqnarray}
\left\{
\begin{array}{rcl}
\displaystyle \frac{\xi}{p} \,+\,\eta & = &\theta \,\xi \\
     & & \\
\beta \,\xi & = & \theta \, \eta     
\end{array}
\right.
\end{eqnarray}
from which we have $\theta^2 \,-\, \theta/p\,-\,\beta\,=\,0$ and we take $\theta=\theta_1$ given by 
 \begin{eqnarray}
\theta_1\,=\,\frac{1/p+\sqrt{1/p^2+4\beta}}{2}\ .
\label{2.8}
\end{eqnarray}
As one can easily check $\theta_1>1$, for all $2<q<2p$, $t>N$ and $1<s<N$. 

\noindent From (\ref{2.6}) we are done provided we prove that for all $k \in \R$ and $r\,<\,dist(y,\partial \Omega)$
there exists $d>0$ such that  
\[
\Phi\left(k+d,\frac{r}{2}\right)=0,
\]
which in turn by (\ref{2.5}) yields 

\[
\left|A^+\left(k+d,\frac{r}{2}\right)\right|\,=\,0 \ .
\]
Next we proceed by using the iterative scheme from the proof of De Giorgi's theorem. For $m \in \N$ set 

\begin{eqnarray*}
\displaystyle 
r_m= \frac{r}{2}\,+\,\frac{r}{2^{m+1}}, \qquad  k_m=k_0+d -\frac{d}{2^m},
\end{eqnarray*}
where the parameter $d>0$ has to be chosen in the sequel and $k_0=k$. The idea is to exploit the inequality (\ref{2.6}) with $r=r_m$ and $\rho = r_{m+1}$ where the sequence  $\{r_m\}_{m\in \N}$ is decreasing so that $B(r_{m+1})\subset B(r_{m})$. On the other hand $\{k_m\}_{m\in \N}$ is increasing, and we set in  (\ref{2.6}) $l=k_{m+1}$ e $k=k_m$. With this choice we obtain 
from (\ref{2.6}) the following inequality 
\begin{multline}
\Phi(k_{m+1},r_{m+1})\,\\
\leq\, \frac{2^{2\frac{(p-1)}{p}(m+2)\xi+2(m+1)\eta}}{r^{2\frac{(p-1)}{p}\xi}\,\,d^{2\eta}}\,\Phi(k_m,r_m)^{\theta} \,M(r_m,k_m,t,p)^{\xi}.
\label{2.10}
\end{multline}

\noindent Now multiply (\ref{2.10}) by $2^{\mu(m+1)}$, $\mu>0$ and set 
\begin{eqnarray}
\Psi_m = 2^{\mu m}  \Phi(k_m,r_m)
\label{2.11}
\end{eqnarray}
to obtain form (\ref{2.10})  
\begin{eqnarray}
\quad\quad  \Psi_{m+1}\,\leq\, \left[\frac{2^{2\frac{(p-1)}{p}(m+2)\xi+2(m+1)\eta}}{r^{2\frac{(p-1)}{p}\xi}\,\,d^{2\eta}}\,2^{\mu m (1-\theta)}\right]\Psi_m^{\theta} \,M(r_m,k_m,t,p)^{\xi}\ .
\label{2.12}
\end{eqnarray}
Let us choose $\mu>0$ to avoid the dependence on $m$ in the first factor in the right hand side of \eqref{2.12}, namely 
\[
\mu=\frac{2\frac{(p-1)}{p}\xi+2\eta}{\theta -1},
\]
and thus (\ref{2.12}) becomes 
\begin{multline*}
\Psi_{m+1}\,\leq\, \frac{2^{2\frac{(p-1)}{p}\xi+2\eta+\mu}}{r^{2\frac{(p-1)}{p}\xi}\,\,d^{2\eta}}\,\Psi_m^{\theta} \,M(r_m,k_m,t,p)^{\xi}\,\\ \leq\, \frac{2^{(21-a)\xi+2\eta+\mu}}{r^{2\frac{(p-1)}{p}\xi}\,\,d^{2\eta}}\,\Psi_m^{\theta} \,M(r,k_0,t,p)^{\xi} \ .
\end{multline*}

\noindent Set 
\[
A\,=\,\frac{2^{2\frac{(p-1)}{p}\xi+2\eta+\mu}}{r^{2\frac{(p-1)}{p}\xi}}\,\,M(r,k_0,t,p)^{\xi},
\]
so that for all $m \in \N$ one has 
\[
\Psi_{m+1}\,\leq\, \frac{A}{d^{2\eta}}\,\Psi_m^{\theta} \ .
\]

\noindent At this point we choose $d>0$ such that 

\begin{eqnarray}
\frac{A}{d^{2\eta}}\,\,\Psi_0^{\theta - 1}\,=\,1,
\label{2.13}
\end{eqnarray}
and by induction on $m \in \N$ we have 
\[
\Psi_m \leq \Psi_0, \quad \text{ for all } m\in\N\ .
\]
\noindent Finally  by (\ref{2.11}) we obtain 
\[
\Phi(k_m,r_m)\leq \frac{1}{2^{\mu m}} \,\Phi(k_0,r) \
\]
and the proof is complete by letting $m\to\infty$.

\end{proof}

\noindent Next we prove the following Harnack type inequality
\begin{teo}\label{cor_c}
Let $u \in W^{1,t}(\Omega)$, $t>N\geq 2$, $and \Omega \subset \R^N$ be an open bounded set with sufficiently smooth boundary $\partial \Omega$.
Then, there exists a constant  $c>0$ which depends only on $N$ such that

\begin{multline}\label{magSup0}
\sup_{B(x_0,\frac{r}{2})} u \,\leq \, \inf_{B(x_0,r)} u \,\\
+\,c\,r^{[(\frac{\xi}{\eta}+\frac{N}{2})(\theta - 1)]} 
\left(\int_{B(x_0,r)} |\nabla u|^t\,\,dx \right)^{\frac{\xi}{\eta}\frac{p-1}{tp}}
\left(\int_{B(x_0,r)} |\nabla u|^2\,\,dx \right)^{\frac{\xi(\theta-1)}{2\eta}}\ .
\end{multline}
\end{teo}

\begin{proof}

Let $B(x_0,r)\subset\Omega$. Set $\displaystyle M=\sup_{B(x_0,r)}u$, $\displaystyle m=\min_{B(x_0,r)}u$, and let 

\[
I_1=\left\{k:\,k\in (m,\,M) :\, \left|\left\{x:\,x\in B(x_0,r),\,u(x)>k \right\}\right|< \frac{|B(x_0,r)|}{2} \right\},
\]

\[
I_2=\left\{k:\,k\in (m,\,M) :\, \left|\left\{x:\,x\in B(x_0,r),\,u(x)\geq k \right\}\right| \geq  \frac{|B(x_0,r)|}{2} \right\}
\]

\noindent If $I_1 \neq \emptyset$  then we prove for all $k \in I_1$ the following 
\begin{multline}\label{magSup1}
\sup_{B(x_0,\frac{r}{2})} u \,\leq \, k \,\\
+\,c\,r^{[(\frac{\xi}{\eta}+\frac{N}{2})(\theta - 1)]}  
\left(\int_{B(x_0,r)} |\nabla u|^t\,\,dx \right)^{\frac{\xi}{\eta}\frac{p-1}{tp}}
\left(\int_{B(x_0,r)} |\nabla u|^2\,\,dx \right)^{\frac{\xi(\theta-1)}{2\eta}}\ .
\end{multline}

\noindent Indeed, by Theorem \ref{teorema1} we have for all $k\in I_1$
\begin{multline}\label{magSup2}
\sup_{B(x_0,\frac{r}{2})} u 
\leq \, k \,+\,c\,r^{-\frac{\xi(p-1)}{\eta p}} |A^+(k,r)|^{\frac{\theta-1}{2}} \\
\cdot\ \left(   \int_{A^+(k,r)} |u -k|^t\,\,dx \,+\,r^t\int_{A^+(k,r)}|\nabla u|^t\,\,dx \right)^{\frac{\xi}{\eta}\frac{p-1}{tp}} \\
\cdot\ \left(\int_{A^+(k,r)} |u-k|^2\,\,dx \right)^{\frac{\xi(\theta-1)}{2\eta}}\ .
\end{multline}

\noindent Since $k \in I_1$ one has 
\[
|\{x:\,(u(x)-k)^+=0 \}|\geq \frac{|B(x_0,r)|}{2},
\]
and apply  Lemma \ref{lemmapoincare}  to the function $(u(x)-k)^+$  to get 

\begin{multline*}
\left(\int_{A^+(k,r)} |u(x)-k|^2\,\,dx \right)^{\frac{1}{2}} =\left(\int_{B(x_0,r)} |(u(x)-k)^+|^2\,\,dx \right)^{\frac{1}{2}}\\
\leq c(N) r
\left(\int_{B(x_0,r)} |\nabla u|^2\,\,dx \right)^{\frac{1}{2}}\ ,
\end{multline*}

\begin{multline*}
\left(\int_{A^+(k,r)} |u(x)-k|^t\,\,dx \right)^{\frac{1}{t}} =\left(\int_{B(x_0,r)} |(u(x)-k)^+|^t\,\,dx \right)^{\frac{1}{t}}\\
\leq c(N) r
\left(\int_{B(x_0,r)} |\nabla u|^t\,\,dx \right)^{\frac{1}{t}} \ .
\end{multline*}

\noindent In the case $I_2 \neq \emptyset$, for all $k \in I_2$ set  $h=-k$ and $v(x)=-u(x)$. Thus $h \in (-M,\,-m)$ and the following holds
\[
\left|\left\{x:\,x\in B(x_0,r): u(x)\geq k \right\} \right| =\left|\left\{x:\,x\in B(x_0,r): -u(x)\leq -k \right\} \right| 
\]
\[
=\left|\left\{x:\,x\in B(x_0,r): v(x)\leq h \right\} \right| \geq \frac{|B(x_0,r)|}{2}.
\] 

\noindent Therefore, the function  $v$ enjoys (\ref{magSup1}), namely 

\begin{multline}\label{magSup3}
\sup_{B(x_0,\frac{r}{2})} v \leq \, h \,\\
+\,c\,r^{[(\frac{\xi}{\eta}+\frac{N}{2})(\theta - 1)]} 
\left(\int_{B(x_0,r)} |\nabla v|^t\,\,dx \right)^{\frac{\xi}{\eta}\frac{p-1}{tp}}
\left(\int_{B(x_0,r)} |\nabla v|^2\,\,dx \right)^{\frac{\xi(\theta-1)}{2\eta}}.
\end{multline}
From 
\[
\sup_{B(x_0,\frac{r}{2})} v \,=\,-\,\inf_{B(x_0,\frac{r}{2})} u
\]
and (\ref{magSup3}) we have 
\begin{multline*}
- \inf_{B(x_0,\frac{r}{2})} u \leq \, -k \,\\
+\,c\,r^{[(\frac{\xi}{\eta}+\frac{N}{2})(\theta - 1)]}  
\left(\int_{B(x_0,r)} |\nabla u|^t\,\,dx \right)^{\frac{\xi}{\eta}\frac{p-1}{tp}}
\left(\int_{B(x_0,r)} |\nabla u|^2\,\,dx \right)^{\frac{\xi(\theta-1)}{2\eta}}.
\end{multline*}

\noindent As a consequence, for all $k \in I_2$ we get 

\begin{multline}\label{magSup4}
k \leq \inf_{B(x_0,\frac{r}{2})} u \,\\+\,c\,r^{[(\frac{\xi}{\eta}+\frac{N}{2})(\theta - 1)]} 
\left(\int_{B(x_0,r)} |\nabla u|^t\,\,dx \right)^{\frac{\xi}{\eta}\frac{p-1}{tp}}
\left(\int_{B(x_0,r)} |\nabla u|^2\,\,dx \right)^{\frac{\xi(\theta-1)}{2\eta}}.
\end{multline}

\noindent Next we distinguish three cases, precisely:

\begin{itemize}

\item[i)] $I_1 \neq \emptyset$ and $I_2 = \emptyset$. In this case any $k\in(m,M)$ belongs to$I_1$,  for which (\ref{magSup1}) which holds for all $k \in I_1$, it holds for $k=m$ as well;

\item[ii)] $I_1 = \emptyset$ e $I_2 \neq  \emptyset$. IIn this case any $k\in(m,M)$ belongs to$I_2$,  and thus  (\ref{magSup4}) which holds for all $k \in I_2$, in particular holds for $k=M$;

\item[iii)]  $I_1 \neq \emptyset$ e $I_2 \neq  \emptyset$. In this case we consider $\inf I_1$ and $\sup I_2$ and it is standard to prove there exists a unique $k_0= \inf I_1 = \sup I_2$ which enjoys both (\ref{magSup1}) and (\ref{magSup4}) and the Theorem follows.
\end{itemize}
\end{proof}

\begin{teo}\label{teoh}
Let $\Omega\subset \R^N$, $N\geq 2$ be open and bounded, and let $x_{max}$ and $x_{min}$ be respectively a local maximum and local minimum for $u \in W^{1,t}(\Omega)$, $t>N$. Then, there exists $h \in \N$, $h=h(\Omega, x_{max}, x_{min})$ such that  
\begin{multline}\label{maggh}
u(x_{max})  \leq u(x_{min}) \,\\+\,c\,\,h\,r^{(\frac{\xi}{\eta}+\frac{N}{2})(\theta-1)}\,\left( \int_{\Omega}|\nabla u(x)|^t\,\,\,dx \right)^{\frac{\xi(p-1)}{tp\eta}}\,\,\,
\left( \int_{\Omega}|\nabla u(x)|^2\,\,\,dx \right)^{\frac{\xi(\theta-1)}{2\eta}}\ ,
\end{multline} 
with $c=c(N)$ provided by Thorem \ref{cor_c} and where in particular $h$ depends only on $dist(x_{max},\partial \Omega)$ and $dist(x_{min},\partial \Omega)$.
\end{teo}

\begin{proof}
Let $r>0 $ be such that:
\begin{itemize}
\item[i)] for all $x \in B(x_{min},r)\subset \Omega$ one has $u(x)>u(x_{min})$;
\item[ii)] $\overline{B(x_{min},r)}\subset \Omega$;
\item[iii)]  $\overline{B(x_{max},r)}\subset \Omega$ \ .
\end{itemize}
Consider the arc $g: [0,1] \longrightarrow \Omega$ such that $g(0)=x_{min}$ and $g(1)=x_{max}$. Let $t_0=0<\dots t_h=1$ be a partition of $[0,1]$ such that setting $x_i=g(t_i)$ one has 
\begin{equation}\label{h2}
B\left(x_i,\frac{r}{2}\right)\cap B\left(x_{i+1},\frac{r}{2}\right)\neq \emptyset, \,\,\,i=0, \dots , h-1 
\end{equation}
and where $r$ is such that $B(x_i,r) \subset \Omega$.

\noindent By Theorem \ref{cor_c} we have

\begin{multline*}
\sup_{B(x_0,\frac{r}{2})}u\leq \,u(x_{min})\,+\,c\, r^{(\frac{\xi}{\eta}+\frac{N}{2})(\theta-1)} \left( \int_{B(x_0,r)}|\nabla u(x)|^t\,\,\,dx \right)^{\frac{\xi(p-1)}{tp\eta}}\\
\cdot \, \left( \int_{B(x_0,r)}|\nabla u(x)|^2\,\,\,dx \right)^{\frac{\xi(\theta-1)}{2\eta}},
\end{multline*}
which we rewrite in the following form 
\begin{equation}\label{h4}
\forall x \in B\left(x_0,\,\frac{r}{2}\right),\,\,\,\,\,\,\,u(x)\leq u(x_{min})+N_0, 
\end{equation}
where we have set for $i=0, \dots ,\,h$
\[
N_i:= c\,\left( \int_{B(x_i,r)}|\nabla u(x)|^t\,\,\,dx \right)^{\frac{\xi(p-1)}{tp\eta}}\,\,\,
\left( \int_{B(x_i,r)}|\nabla u(x)|^2\,\,\,dx \right)^{\frac{\xi(\theta-1)}{2\eta}}\ .
\]
Now inequality (\ref{h4}) in particular holds for $$x \in B\left(x_1,\frac{r}{2}\right)\cap B\left(x_0,\frac{r}{2}\right)$$
and thus 
\begin{equation}\label{h5}
\inf_{B\left(x_1, \frac{r}{2}\right)} u \leq u(x)\leq u(x_{min}) + N_0 \ .
\end{equation}
By applying iteratively Theorem \ref{cor_c} we end up with 
\begin{equation*}
\sup_{B(x_h,\,\frac{r}{2})} u\,\leq \, u(x_{min})\,+ \,N_h\,+\,\cdots\,+ \,N_1\,+\,N_0 \ . 
\end{equation*}
which completes the proof. 

\end{proof}
\begin{oss} One may wonder what happens if in the construction of Theorem \ref{teoh} we consider a sequence of balls with increasing radius e center converging to a point on the boundary of $\Omega$. For this purpose consider $\{x_n\}_{n \in \N}\subset\Omega$ converging to a point $x_{\infty}\in \partial \Omega$.
Consider balls of center $x_n$ and radius $r_n$ such that: 
\begin{itemize}
\item[(i)] $B(x_n,r_n) \subset \Omega$;
\item[(ii)] $r_n <\,dist(x_n,\partial \Omega)$ $=$ $dist(x_n,x_{\infty})$;
\item[(iii)] $\displaystyle B(x_n,\frac{r_n}{2})\cap B(x_{n+1},\frac{r_{n+1}}{2})\neq \emptyset$. 
\end{itemize}

\noindent Applying to this sequence the reasoning carried out in the proof of Theorem \ref{teoh} where $x_0=x_{max}$, we get

\[
u(x_{max}) \leq u(x_{\infty}) + c \sum_{n=0}^{\infty} r_n^{(\frac{\xi}{\eta}+\frac{N}{2})(\theta -1)} 
\left( \frac{1}{\gamma}\int_{\Omega_1}f(x)\,\,dx \right)^{\frac{\xi}{2\eta}(\theta -1)}  u(x_{max}).
\]

\noindent We would get a contradiction if the above series converge. Actually as we are going to see this is not the case. Consider $B(x_n,\frac{r_n}{2})$ and $ B(x_{n+1},\frac{r_{n+1}}{2})$ and let 
 $C \in  B(x_n,\frac{r_n}{2})\cap B(x_{n+1},\frac{r_{n+1}}{2}) $ and $D$ its projection on the segment with endpoints $A= x_n$ and $B=x_{n+1}$. 
 Set $AD= \rho_n$, $DB= \rho_{n+1}$, so that considering the triangle $ADC$ and $CDB$ one has 
 $\displaystyle  \frac{r_n^2}{4}-\rho_n^2=\frac{r_{n+1}^2}{4}-\rho_{n+1}^2$,  and then 
 \[
\frac{\frac{r_n}{4}+\rho_n}{\frac{r_{n+1}}{4}+\rho_{n+1}} =\frac{\frac{r_{n+1}}{4}-\rho_{n+1}}{\frac{r_{n}}{4}-\rho_{n}}. 
 \]
We can apply Kummer's test to the series with general terms $\displaystyle a_n=\left(\frac{r_n}{4}+\rho_n\right)^a$
and $\displaystyle b_n=\left(\frac{r_n}{4}-\rho_n\right)^a$, $a>0$, from which since $\displaystyle \frac{a_n}{a_{n+1}}=\frac{b_{n+1}}{b_n}$, for all  $n \in \N$,  and $\displaystyle \sum_{n=0}^{\infty} \frac{1}{b_n}=+\infty$ we obtain 
$\displaystyle \sum_{n=0}^{\infty}a_n = + \infty$. From $a_n < r_n^a$ we have $\displaystyle \sum_{n=0}^{\infty} r_n^a = + \infty$ .

\end{oss}

\section{Towards the Positivity Preserving Property}

\noindent Next we apply the results so far obtained to prove the strong maximum principle for the biharmonic operator perturbed by the Laplacian for compactly supported data. As we are going to see, here it comes for the first time the restriction on the Euclidean dimension $N<4$ and the fact that we deal with the solution to a PDE. Precisely, this Section is devoted to prove the following 
\begin{teo}\label{teo1pos}
Let $\Omega \subset \R^N$, $N=2,3$ be an open and bounded set, with sufficiently smooth boundary and which enjoys the interior sphere condition. Let $u\in W^{4,2}\cap H^2_0(\Omega)$ be solution to 
\begin{equation}\label{1pos}
\Delta^2u(x)- \gamma \Delta u(x)=f(x),\,\,\,\,x\in \Omega,
\end{equation}
where $\gamma >0$, $f\in L^2(\Omega)$, $f\geq 0$ in $\Omega$ and $|\{x: f(x)>0 \}|>0$. Moreover, $f(x)=0$ on $\Omega \setminus \Omega_1$, with  $\Omega_1$
 a bounded subset of $\Omega$ such that $dist(\partial\Omega_1, \partial \Omega)>0$. Then, there exists $\gamma_0>0$ such that for all $\gamma > \gamma_0$ the solution to \eqref{1pos} satisfies $u(x)>0$, for all $x \in \Omega$.
\end{teo}

\noindent Assuming the hypotheses of Theorem \ref{teo1pos} we have the following preliminary lemmas: 
\begin{lemma} The following holds true
\begin{equation}\label{2pos}
\sup_{\Omega_1}u \,>\,0\ . 
\end{equation}
\end{lemma}
\begin{proof}
By multiplying \eqref{1pos} by $u$ and integrating by parts
\begin{multline}\label{dis2pos}
\int_{\Omega}|\Delta u(x)|^2\,\,\,dx\,+\,\gamma\,\int_{\Omega} |\nabla u(x)|^2\,\,\,dx \,\\
=\,\int_{\Omega} f(x)\,u(x)\,\,\,dx\,\leq \,\sup_{\Omega_1}u\,\,\,\int_{\Omega_1} f(x)\,\,\,dx
\end{multline}
\end{proof}
\noindent In order to apply the Harnack inequality established in Section \ref{harnack} we next estimate first order derivatives of the solution to \eqref{1pos}. Though from one side elliptic regularity yields enough summability, on the other side we need estimates which are uniform with respect to the parameter $\gamma$, and for this reason we restrict ourself to dimensions $N<4$. 
\begin{lemma}\label{gradient_est}
There exists a constant $c=c(N)>0$  which does not depend on $\gamma$ in (29) such that
\[
\|\nabla u \|_{L^t(\Omega)} \leq  cd_{\Omega}^{\frac{2}{t}(3-N)} \int_{\Omega} f(x)\,u(x)\,\,dx,
\]  
for any $t>2$ when $N=2$ and for $t=6$ when $N=3$.
\end{lemma}
\begin{proof}
Since $u=\nabla u=0$ su $\partial \Omega$, one has 
\begin{equation*}
\int_{\Omega}|\Delta u(x)|^2\,\,dx\,= \,\sum_{i,j=1}^n \int_{\Omega}|D_{ij}u(x)|^2\,\,dx\,=\,\int_{\Omega} \|D^2u(x)\|^2\,\,dx\ .
\end{equation*}
By Sobolev's embedding and from \eqref{dis2pos}, when $N=3$ and  $t=6$ we have 
\begin{multline*}
\|\nabla u \|_{L^t(\Omega)} \leq \frac{c}{d_{\Omega}}\,\|\nabla u \|_{L^2(\Omega)} + c \|D^2 u \|_{L^2(\Omega)}\\
\leq c \|D^2 u \|_{L^2(\Omega)}
= c\|\Delta u \|_{L^2(\Omega)} \leq c \left(\int_{\Omega} f(x)\,u(x)\,\,dx\right)^{\frac{1}{2}}\ .
\end{multline*}
Similarly when $N=2$ and $t \geq 1$ we obtain 
\begin{multline*}
\|\nabla u \|_{L^t(\Omega)} \leq \frac{c}{d_{\Omega}^{1-\frac{2}{t}}}\|\nabla u \|_{L^2(\Omega)} + c d_{\Omega}^{\frac{2}{t}} \|D^2 u \|_{L^2(\Omega)}\\
\leq 
c d_{\Omega}^{\frac{2}{t}} \|D^2 u \|_{L^2(\Omega)}
= c d_{\Omega}^{\frac{2}{t}}\|\Delta u \|_{L^2(\Omega)} \leq c d_{\Omega}^{\frac{2}{t}}\left(\int_{\Omega} f(x)u(x)\,dx\right)^{\frac{1}{2}}\ .
\end{multline*}
\end{proof}

\begin{proof}[Proof of Theorem \ref{teo1pos}]
\noindent Let $x_{max}$ be an absolute maximum point for $u$ in $\overline{\Omega_1}$ and $x_{min}$ a local minimum for $u$ in $\Omega$. Set 
$$\displaystyle \mathbf{a}=\frac{\xi(\theta -1)}{2\eta}+\frac{\xi(p-1)}{2\eta p},\: \displaystyle \mathbf{b}=\left(\frac{\xi}{\eta}+\frac{N}{2}\right)(\theta -1),\: \displaystyle \mathbf{c}= \frac{\xi}{\eta} \frac{p-1}{p}\ .$$ 
From (31), Theorem 3.2 and Lemma 4.2 we have

\begin{equation}\label{dis4pos}
\qquad \qquad  u(x_{max}) \leq u(x_{min}) + c h \,r^{\mathbf{b}} \,\,d_{\Omega}^{[\frac{2}{t}(3-N)]\mathbf{c}} \,\,\frac{ \left( \int_{\Omega} f(x)u(x)\,\,dx \right)^{\mathbf{a}}}{\gamma^{\frac{\xi}{4 \eta}(\theta-1)}}
\end{equation}

\noindent where $\displaystyle \mathbf{a}=\frac{\xi(\theta -1)}{2\eta}+\frac{\xi(p-1)}{2\eta p} <1$. If $\displaystyle  \sup_{\Omega_1} u\geq 1$ then we have

\[
\qquad \qquad u(x_{max}) \leq u(x_{min}) + c h \,r^{\mathbf{b}} \,\,d_{\Omega}^{[\frac{2}{t}(3-N)]\mathbf{c}} \,\,\frac{ \left( \int_{\Omega} f(x)\,\,dx \right)^{\mathbf{a}} u(x_{max})}{\gamma^{\frac{\xi}{4 \eta}(\theta-1)}}
\ .
\]

\noindent The thesis follows as $\gamma$ is large enough. If $\displaystyle \sup_{\Omega_1}u < 1$, let $k >0$ be such that $\displaystyle k\sup_{\Omega_1} u \geq 1$. Set $w_k(x):=k\,u(x)$, which satisfies 
\begin{eqnarray}
\left\{
\begin{array}{l}
w_k\in W^{4,2}\cap H^{2}_0(\Omega_s)\\
\\
\Delta^2w_k(x)- \gamma \Delta w_k(x)=k\,f(x),\,\,\,\,x\in \Omega,
\end{array}
\right.
\label{kpos}
\end{eqnarray}
Peforming the change of variable $x=sy$, with $s>0$, $v_k(y)=w_k(sy)$, $g(y)=f(sy)$, $y_{min} = \frac{x_{min}}{s}$, $y_{max}=\frac{x_{max}}{s}$, we obtain 

\begin{eqnarray}
\left\{
\begin{array}{l}
v_k\in W^{4,2}\cap H^2_0(\Omega)\\
\\
\Delta^2v_k(y)- \gamma\,s^{-2}\, \Delta v_k(y)=\,s^{-4}\,k\,g(y),\,\,\,\,y\in \Omega_s,
\end{array}
\right.
\label{kspos}
\end{eqnarray}
\noindent where $\displaystyle \Omega_s = \{y: \,\,\,y=x/s, \,\,\,\,x \in \Omega\}$. Next apply \eqref{dis4pos} to the solution of \eqref{kspos} to get 
\begin{multline}\label{disvkspos}
v_k(y_{max}) \leq v_k(y_{min}) \\
+ c h \,\left(\frac{r}{s}\right)^{\mathbf{b}} \,\,\left(\frac{d_{\Omega}}{s}\right)^{[\frac{2}{t}(3-N)]\mathbf{c}} \,\,\frac{ \left( \int_{\Omega_s} g(x)\,\,dx \right)^{\mathbf{a}} v_k(y_{max})}{\gamma^{\frac{\xi}{4 \eta}(\theta-1)}}
\frac{k^{\mathbf{a}}}{s^{3\mathbf{a}}}\ .
\end{multline}
With respect to the original variables it reads as follows 
\begin{multline}\label{dis5pos}
u(x_{max}) \leq u(x_{min}) \\
+ c h \,\left(\frac{r}{s}\right)^{\mathbf{b}} \,\,\left(\frac{d_{\Omega}}{s}\right)^{[\frac{2}{t}(3-N)]\mathbf{c}} \,\,\frac{\left( \int_{\Omega} f(x)\,\,dx \right)^{\mathbf{a}} u(x_{max}) }{\gamma^{\frac{\xi}{4 \eta}(\theta-1)}}
\frac{k^{\mathbf{a}}}{s^{(N+3)\mathbf{a}}} \ . 
\end{multline}

\noindent Let us now observe that thanks to the interior sphere condition, the number $h$ of balls covering he path 
from $y_{max}$ to $y_{min}$ does not depend on the parameter $s$. 
The same happens for the parameter $k$.
Thus we choose the parameter $s$ such that  
\begin{equation}\label{hks}
\frac{h\,k^{\mathbf{a}}}{s^{\mathbf{b}+[\frac{2}{t}(3-N)]\mathbf{c}+(N+2)\mathbf{a}}} \,=\, 1,
\end{equation}
namely the thesis of the Theorem follows for all
\begin{equation}\label{gamma0}
 \gamma > c^{\frac{2\eta}{\xi (\theta -1)}}  d_{\Omega}^{\frac{2\eta  }{\xi (\theta -1)}\{[\frac{2}{t}(3-N)]\mathbf{c}+\mathbf{b}\}}
\left(\int_{\Omega}f(x)\,\,\,dx \right)^{\mathbf{a}\frac{2\eta  }{\xi (\theta -1)}}, 
\end{equation}
and thus $\gamma_0$ is the right hand side of \eqref{gamma0} with optimal constant $c$. When $\gamma=\gamma_0$ we just get the weak inequality $u\geq 0$. 
\end{proof}

\section{The validity of the strong maximum principle for higher order elliptic operators}\label{higher_order}
\noindent In this Section we first prove Theorem \ref{main} for which we have to remove the restriction to compactly supported data of Theorem \ref{teo1pos}. Then, we will extend the result obtained to polyharmonic operators and to more general uniformly elliptic operators of any even order with constant coefficients. 
\begin{proof}[Proof of Theorem \ref{main}]

\noindent Consider the following family of sets $\{\Omega_m\}_{m \in \N}$ such that for all $m \in \N$ satisfy:

\begin{itemize}
\item[i)] $\overline{\Omega}_m \subset \Omega_{m+1}\subset \overline{\Omega}_{m+1}\subset \Omega$;
\item[ii)] $\cup_{m=1}^{\infty} \Omega_m =\Omega$;
\item[iii)] $|\{x:\,\,f>0\}|>0\cap \Omega_1 \neq \Omega_1$;
\item[iv)] $ dist(\partial \Omega_m, \partial \Omega) \longrightarrow 0$ as  $m\to\infty$. 
\end{itemize}

\noindent Let $\chi_m$ be the characteristic function of $\Omega_m$
\begin{eqnarray*}
\chi_m(x)=
\left\{
\begin{array}{l}
1,\,\,\,x \in \Omega_m,\\
\\
0,\,\,\,x \notin \Omega_m.
\end{array}
\right.
\end{eqnarray*}
and set  
\begin{eqnarray}
g_m(x):=\frac{1}{S(x)}\,\,\frac{\chi_m(x)}{m^2}\,f(x), \,\,\,\,x \in \Omega,
\label{7pos}
\end{eqnarray}
where 
\[
S(x)=\sum_{m=1}^{+\infty}  \frac{\chi_m(x)}{m^2},
\]
converges pointwise on  $\Omega$. Moreover, notice that  $g_m \in L^2(\Omega)$.

\noindent Next consider the following problems 

\begin{eqnarray}
\left\{
\begin{array}{l}
u_m\in W^{4,2}\cap H^2_0(\Omega)\\
\\
\Delta^2u_m(x)- \gamma \Delta u_m(x)=g_m(x),\,\,\,\,x\in \Omega,
\end{array}
\right.
\label{mpos}
\end{eqnarray}
where by construction $g_m(x)=0$ for $x \in \Omega \setminus \Omega_m$ and thus by Theorem \ref{teo1pos} there exists $\gamma_m>0$ such that for all $\gamma > \gamma_m$, one has $u_m(x)>0$, for all $x \in \Omega$, $m \in \N$.

\noindent It is crucial here that by (\ref{hks}) and (\ref{gamma0}) the parameter $\gamma_m$ does not depend on $h$, namely does not depend on the distance of the maximum point of $u_m$ from the boundary (recall the proof of Theorem \ref{teoh}). Indeed, this prevents $\gamma_m$ to blow up and actually remain bounded since from \eqref{gamma0}
\begin{multline*}
\gamma_m =  c^{\frac{4\eta}{\xi (\theta -1)}} 
 d_{\Omega}^{\frac{4\eta}{\xi (\theta -1)}\{[\frac{2}{t}(3-N)]\mathbf{c}+\mathbf{b}\}}
\left(\int_{\Omega_m}g_m(x)\,\,\,dx \right)^{\mathbf{a}\frac{4\eta}{\xi (\theta -1)}}
\\
\leq  c^{\frac{4\eta}{\xi (\theta -1)}}  d_{\Omega}^{\frac{4\eta  }{\xi (\theta -1)}\{[\frac{2}{t}(3-N)]\mathbf{c}+\mathbf{b}\}}
\left(\int_{\Omega}f(x)\,\,\,dx \right)^{\mathbf{a}\frac{4\eta  }{\xi (\theta -1)}}
=\gamma_{\infty}
\end{multline*}
Therefore, for all $\gamma > \gamma_{\infty}$ and for all $m \in \N$ one has 
\begin{equation}\label{8pos}
u_m(x)>0,\quad x\in\Omega  
\end{equation}

\noindent Finally, we prove that the function 
\begin{equation}
\displaystyle
v(x)= \sum_{m=1}^{\infty} u_m(x)
\label{9pos}
\end{equation}
solves the following 
\begin{eqnarray}
\left\{
\begin{array}{l}
v\in W^{4,2}\cap H^2_0(\Omega)\\
\\
\Delta^2v(x)- \gamma \Delta v(x)=f(x),\,\,\,\,x\in \Omega
\end{array}
\right.
\label{10pos}
\end{eqnarray}
and thus by (\ref{8pos}) we conclude that for all $\gamma > \gamma_{\infty}$ and for all $x \in \Omega$ one has 
\[
v(x)>0 \ .
\]

\noindent By uniqueness of the solution to the Dirichlet problem (\ref{1pos})  the Theorem follows.
Hence, it remains to show that $v_m\to v\in W^{4,2}\cap H^2_0(\Omega)$ which is a solution to (\ref{10pos}).

\noindent Set 

\[
f_m = \sum_{i=1}^m g_i,\qquad v_m= \sum_{i=1}^m u_i \ .
\]
By Lebesgue's dominated convergence $f_m\to f$  in $L^2(\Omega)$ and notice that  $v_m$  solves the following  

 \begin{eqnarray}
\left\{
\begin{array}{l}
v_m\in W^{4,2}\cap H^2_0(\Omega)\\
\\
\Delta^2v_m(x)- \gamma \Delta v_m(x)=f_m(x),\,\,\,\,x\in \Omega \ .
\end{array}
\right.
\label{mpos}
\end{eqnarray}
Thus for all $m$, $l\in \N$ we have 
\begin{equation}\label{mpos2}
\begin{cases}
v_m-v_l\in W^{4,2}\cap H^2_0(\Omega)&\\
&\\
\Delta^2[v_m(x)-v_l(x)]- \gamma \Delta [v_m(x)-v_l(x)]=f_m(x)-f_l(x)
\end{cases}
\end{equation}
and multiplying by $v_m-v_l$ and integrating by parts we get 

\[
\int_{\Omega} |\Delta [v_m(x)-v_l(x)]|^2\,\,\,dx\,\leq \int_{\Omega}|f_m(x)-f_l(x)|^2\,\,\,dx
\]
which together with the equation \eqref{mpos} yields 

\[
\int_{\Omega} |\Delta ^2[v_m(x)-v_l(x)]|^2\,\,\,dx\,\leq \,c\,(14\gamma^2+2)\int_{\Omega}|f_m(x)-f_l(x)|^2\,\,\,dx \ .
\]

\noindent Thus $\{v_m\}$ is a Cauchy sequence in $W^{4,2} (\Omega)$ which converges to $v\in W^{4,2} (\Omega)$, the solution to \eqref{10pos}.
\end{proof}
\noindent What we have seen so far naturally extends to polyharmonic operators of any order and more in general to uniformly elliptic operators of any even order as established in the following 
\begin{cor}\label{higher_order_cor}
Let $u \in W^{2m,2}\cap W^{m,2}_0(\Omega)$, $m\geq 2$ be the solution to the following equation 
\begin{equation}\label{higher_eq}
(-1)^m {\mathcal A}_{2m}(D)u(x)-\gamma {\mathcal A}_2(x,D) u(x)= f(x), \,\,\,\,x \in \Omega,
\end{equation}
where $f \in L^{2}(\Omega)$,
\[
{\mathcal A}_{2m}(D)= \sum_{|\alpha|=|\beta|=m}a_{\alpha\beta}D^{\alpha+\beta}
\]
and
\[
{\mathcal A}_{2}(x,D)= \sum_{i,j=1}^n D_i[a_{ij}(x)D_j],
\] 

\noindent are uniformly elliptic operators on $\Omega$, namely there exist $\nu_{m}>0$ and $\nu_1>0$ such that for all $\xi \in \R^N$ and  $x \in \Omega$
\[
\nu_{m} \|\xi\|^{2m}\leq \sum_{|\alpha|=|\beta|=m}a_{\alpha\beta}\xi^{\alpha+\beta}, \quad  \nu_1 \|\xi\|^2  \leq  \sum_{i,j=1}^n a_{ij}(x)\xi_i\xi_j,
\]
with $a_{\alpha\beta}\in \R$ and $a_{ij}(x)\in L^{\infty}(\Omega)$. Then, there exists $\gamma_0>0$ such that for all $\gamma>\gamma_0$
one has $u(x)>0$ for all $x\in\Omega$. 
\end{cor}
\begin{proof}
We have to estimate intermediate derivatives of suitable order avoiding the dependance on $\gamma$. Multiplying the equation \eqref{higher_eq} by $u$ and integrating by parts we get  
\begin{multline*}
\int_{\Omega}  \sum_{|\alpha|=|\beta|=m}a_{\alpha\beta}D^{\alpha} u(x)\,D^{\beta}u(x)\,\,dx\,+\,\gamma\,
\int_{\Omega} \sum_{i,j=1}^n \,a_{ij}(x)D_ju(x)D_iu(x)\,\,\,dx\,\\ \leq \,\int_{\Omega} f(x)\,u(x)\,\,\,dx
\leq \|f\|_{L^2(\Omega)} \|u\|_{W^m_0(\Omega)} \ .
\end{multline*}
By the ellipticity condition and G{\aa}rding's inequality one has 
\[
\nu_m \|u\|^2_{W^m_0(\Omega)} + \gamma \nu_1  \|u\|^2_{W^{1,2}_0(\Omega)} \leq \int_{\Omega} |f(x)|\,|u(x)|\,\,\,dx \
\]
together with Poincar\'e's inequality 

\[
\nu_m \|u\|^2_{W^m_0(\Omega)}  \leq c(N,\Omega)\int_{\Omega} |f(x)|^2\,\,\,dx \ .
\]
   
\noindent We conclude by the Sobolev embedding theorem as follows:
\begin{itemize}
\item If $N\leq 2(m-1)$ one has $\nabla u \in L^t(\Omega)$, for all $t\geq 1$ and in particular for $t>N$ and  
\[
\|\nabla u\|_{L^t(\Omega)} \leq c  \|u\|_{W^{m,2}(\Omega)} \leq \|f\|_{L^2(\Omega)}\ ;
\]
\item If $N=2m-1$ one has $\nabla  u \in L^t(\Omega)$ with $t=4m-2$ and
\[
\|\nabla u\|_{L^t(\Omega)} \leq c  \|u\|_{W^{m,2}(\Omega)} \leq c  \|f\|_{L^2(\Omega)}\ .
\]
\end{itemize}
\end{proof}
\noindent It is well known from \cite{GazzolaGrunauSweers10} that the positivity preserving property of the ball for polyharmonic operators carries over to small deformations of the ball.  Actually on those domains what we have proved yields the positivity preserving property of the $\gamma$-perturbed polyharmonic operator for all $\gamma\geq 0$. For simplicity let us state the result in the case of the biharmonic operator:
\begin{cor}\label{corf1}
Let $\Omega\subset\R^N$ be an open bounded domain such that for all $f \in L^2(\Omega)$ with $f \geq 0$ and  $|\{x\in\Omega:\, f(x)=0\}|=0$, the solution $u \in W^{4,2}\cap W^{2,2}_0(\Omega)$ to 
\begin{eqnarray}
\Delta^2 u =f
\label{1} 
\end{eqnarray}
enjoys  $u(x) >0$, a.e. in $\Omega$. If there exists 
$\gamma_0 >0$ such that the solution  $v \in W^{4,2}\cap W^{2,2}_0(\Omega)$ to 
\begin{eqnarray}
\Delta^2 v - \gamma_0 \Delta v =f
\label{2}
\end{eqnarray}
enjoys   $v(x) >0$, a.e. in $\Omega$, then 
for all  $\gamma\in [0, \gamma_0]$ the solution $w \in W^{4,2}\cap W^{2,2}_0(\Omega)$ to $\Delta^2 w_{\gamma} - \gamma \Delta w_{\gamma} =f$ enjoys
   $w_{\gamma}(x) >0$, a.e. in $\Omega$.
\end{cor}

\begin{proof}

Set $w_{\tau} = \tau v + (1-\tau) u $, $\gamma_{\tau}= \tau \gamma_0$, hence $\Delta^2 w_{\tau}- \gamma_{\tau} \Delta w_{\tau}=f$.
For $\tau = 0$, $w_{\tau}$ is a solution to (\ref{1}) whence for $\tau =1$, $w_{\tau}$ enjoys (\ref{2}) and then $w_{\tau}>0$  a.e. in $\Omega$ for all $\tau \in [0,1]$.
By uniqueness of the Dirichlet problem $w_{\tau}=w_{\gamma}$ and the claim follows. 

\end{proof}

\noindent From Corollary \ref{corf1} we also have 

\begin{cor}
Let $\Omega\subset\R^N$ be an open bounded domain such that for all $f \in L^2(\Omega)$ with  $f \geq 0$ and $|\{x\in\Omega:\, f(x)=0\}|=0$, the solution $u \in W^{4,2}\cap W^{2,2}_0(\Omega)$ to $ \Delta^2 u =f$
enjoys    $u(x) >0$, a.e. in $\Omega$. Then,  if $N=2,\,3$, we have for all $\gamma \in [0, \,+\infty)$ that $w_{\gamma}>0$.  
\end{cor}





\end{document}